\documentclass[11pt]{article}

\usepackage{hyperref} 
\usepackage{color}
\usepackage{amsmath}
\usepackage{amssymb}
\usepackage{amsthm}
\usepackage{stmaryrd,graphicx}

\usepackage{comment}
\usepackage{cases}

\usepackage{indentfirst}

\newtheorem{thm}{Theorem}[section]
\newtheorem{lemma}[thm]{Lemma}
\newtheorem{prop}[thm]{Proposition}

\newtheorem{cor}[thm]{Corollary}
\theoremstyle{definition}

\newtheorem{rem}[thm]{Remark}

\newcommand{\E}[1]{\mathbb{E}\left[#1\right]}
\newcommand{\p}{\mathbb{P}}
\renewcommand{\P}{\mathbb{P}}
\newcommand{\Var}{{\mathrm{Var}}\,}

\newcommand{\Z}{\mathcal{Z}}
\newcommand{\N}{\mathcal{N}}

\newcommand{\new}[1]{{\color{blue} #1}}

\def\br#1{\left(#1\right)}
\def\brb#1{\left[#1\right]}

\numberwithin{equation}{section}

\author{\uppercase{\footnotesize{Dmitry Beliaev}}
	\footnotesize{ and }
\uppercase{\footnotesize{Riccardo W. Maffucci}}}
\title{\normalsize{\uppercase{\bf{
Intermediate and small scale limiting theorems for random fields}}}}
\date{}
\newcommand{\Addresses}{{
  \bigskip
  \footnotesize
  
  D.~Belyaev, \textsc{Mathematical Institute, University of Oxford, Woodstock Road Oxford OX2 6GG, UK}\par\nopagebreak
  \texttt{dmitry.belyaev@maths.ox.ac.uk}
  
  \bigskip
  
  R.W.~Maffucci, \textsc{Mathematical Institute, University of Oxford, Woodstock Road Oxford OX2 6GG, UK}\par\nopagebreak
  \texttt{riccardo.maffucci@maths.ox.ac.uk}
}}

\makeatletter
\newcommand{\fixed@sra}{$\vrule height 2\fontdimen22\textfont2 width 0pt\shortrightarrow$}
\newcommand{\shortarrow}[1]{%
  \mathrel{\text{\rotatebox[origin=c]{\numexpr#1*45}{\fixed@sra}}}
}
\makeatother

\begin{document}
\maketitle
\begin{abstract}
In this paper we study the nodal lines of random eigenfunctions of the Laplacian on the torus, the so called `arithmetic waves'. To be more precise, we study the number of intersections of the nodal line with a straight interval in a given direction. We are interested in how this number depends on the length and direction of the interval and the distribution of spectral measure of the random wave. We analyse the second factorial moment in the short interval regime and the persistence probability in the long interval regime. We also study relations between the Cilleruelo and Cilleruelo-type fields. We give an explicit coupling between these fields which on mesoscopic scales preserves the structure of the nodal sets with probability close to one.



\end{abstract}

\section{Introduction}
\subsection{The random toral wave ensemble}
In recent years the merge of ideas from several branches of mathematics has brought by the study of {\bf nodal lines}. For a real- or complex-valued function $\Phi$ defined on a smooth compact surface $\mathcal{S}$, the {\em nodal set} is
\begin{equation}
\label{nodset}
\{x\in\mathcal{S} : \Phi(x)=0\}.
\end{equation}
Let $\Delta$ be the Laplacian operator on $\mathcal{S}$. The study of functions $\Phi$ satisfying the Helmholtz equation
\begin{equation*}
(\Delta+E)\Phi=0
\end{equation*}
with eigenvalue (or `energy') $E>0$, especially in the high energy limit $E\to\infty$, has several applications in PDEs and in physics, e.g. the study of waves. Here the nodal lines remain stationary during membrane vibrations.

Our starting point is the ensemble of `arithmetic random waves', first introduced in 2007 by Oravecz, Rudnick and Wigman \cite{orruwi}. These waves are the {\bf random Gaussian Laplace toral eigenfunctions}, defined as follows. On the torus $\mathbb{T}^2=\mathbb{R}^2/\mathbb{Z}^2$, the Laplace eigenvalues are of the form $E_m=4\pi^2m$ where $m$ is the sum of two integer squares. Let
\begin{equation}
\label{lp}
\Lambda=\Lambda_m=\{\lambda\in\mathbb{Z}^2 : |\lambda|^2=m\}
\end{equation}
be the set of lattice points on the circle $\sqrt{m}\mathbb{S}^1$. For the eigenvalue $E_m$ the collection of exponentials
\begin{equation*}
\{e^{2\pi i\langle\lambda,x\rangle}\}_{\lambda\in\Lambda_m}
\end{equation*}
forms a basis for the relative eigenspace. Therefore, all the (complex-valued) eigenfunctions corresponding to the eigenvalue $4\pi^2m$ have the expression
\begin{equation}
\label{eigd}
\Phi(x)=
\sum_{\lambda\in\Lambda}
g_{\lambda}
e^{2\pi i\langle\lambda,x\rangle},
\end{equation}
with $g_{\lambda}$ Fourier coefficients. The wavelength is $1/\sqrt{m}$. The eigenspace dimension is the number of representations $r_2(m)=|\Lambda_m|$ of $m$ as the sum of $2$ perfect squares. We may now give the definition of arithmetic random waves
\begin{equation}
\label{arwd}
\Psi_m(x)=\frac{1}{\sqrt{r_2(m)}}
\sum_{\lambda\in\Lambda}
a_{\lambda}
e^{2\pi i\langle\lambda,x\rangle},
\qquad\quad
x\in\mathbb{T}^2,
\end{equation}
where $a_{\lambda}$ are complex standard Gaussian random variables \footnote{These are defined on a probability space $(\Omega,\mathcal{F},\mathbb{P})$, and $\mathbb{E}$ denotes expectation with respect to $\mathbb{P}$.} (in the sense that $\mathbb{E}[a_{\lambda}]=0$ and $\mathbb{E}[|a_{\lambda}|^2]=1$). The $a_{\lambda}$ are independent save for the relations $a_{-\lambda}=\overline{a_{\lambda}}$, making the eigenfunction \eqref{arwd} real-valued. We may equivalently write \footnote{The notation $A\sim \N(\mu,\sigma^2)$ means that $A$ is a (real) Gaussian r.v. of mean $\mu$ and variance $\sigma^2$.}
\begin{equation*}
a_{\lambda}=b_\lambda+ic_\lambda,\qquad\qquad b_\lambda,c_\lambda\sim \N(0,1/2)
\end{equation*}
and the $b_\lambda,c_\lambda$ are independent save for $b_{-\lambda}=b_{\lambda}$ and $c_{-\lambda}=-c_{\lambda}$. The coefficients (such as the $1/\sqrt{r_2(m)}$ in \eqref{arwd}) are always chosen so that the random fields are unit variance.

We are interested in the behaviour of arithmetic waves as $m\to \infty$. In this context, it is natural to re-scale the function so that its wavelength becomes $1$. Namely, we consider 
\begin{equation}
\label{rescaled_wave}
\Psi_m(x)=\frac{1}{\sqrt{r_2(m)}}
\sum_{\lambda\in\Lambda_m}
a_{\lambda}
e^{2\pi i\langle\lambda/\sqrt{m},x\rangle},
\qquad\quad
x\in\sqrt{m}\mathbb{T}^2.    
\end{equation}

The {\em spectral measure}  $\nu_m$ of the rescaled function  purely atomic and is given by
\begin{equation}
\label{nu}
\nu_m:=\frac{1}{r_2(m)}\sum_{\lambda\in\Lambda}\delta_{\lambda/\sqrt{m}}
\end{equation}
and its support is 
\[
\frac{1}{\sqrt{m}}\Lambda_m\subset\mathbb{S}^{1}.
\]

The behaviour of arithmetic waves can be described in terms of the $\nu_m$ or, equivalently, in terms of its support. These measures change in a complicated way as $m\to\infty$. In particular, they do not converge. On the other hand, for a   {\em generic} (i.e., density one) sequences of energy levels $\{m_k\}_k$,  the measures $\nu_{m_k}$ converge weak-* to the uniform measure on $\mathbb{S}^1$: 
\begin{equation}
\label{weak}
\nu_{m_k}\Rightarrow\frac{d\theta}{2\pi}.
\end{equation}
See \cite{erdhal, equidi} for details.

To the other extreme, Cilleruelo showed that there exist (`thin' i.e. density zero) sequences of energy levels $\{m_k\}_k$ such that all the lattice points lie on arbitrarily short arcs:
\begin{equation}
\label{cilmeas}
\nu_{m_k}\Rightarrow\frac{1}{4}(\delta_{\pm 1}+\delta_{\pm i}),
\end{equation}
with $\nu_m$ as in \eqref{nu}. Indeed, Cilleruelo proved the following result.
\begin{prop}[{\cite[Theorem 2]{ciller}}]
	For every $\epsilon>0$ and for every integer $k$, there exists a circle $\sqrt{m}\mathbb{S}^1$ such that all the lattice points of $\Lambda_m$ are on the arcs $\sqrt{m}e^{i\pi/2(t+\theta)}$, $|\theta|<\epsilon$, $t=0,1,2,3$, and $|\Lambda_m|>k$.
\end{prop}

The limiting measure in \eqref{cilmeas} is called the {\it Cilleruelo measure} in the literature \cite{ciller, rudwig, krkuwi, roswig}. Figure \ref{pic3} represents an arithmetic random wave in case of spectral measure supported on $8$ points that are close to the support of the Cilleruelo measure.

\begin{figure}
	\begin{center}
		\includegraphics[width=0.8\textwidth]{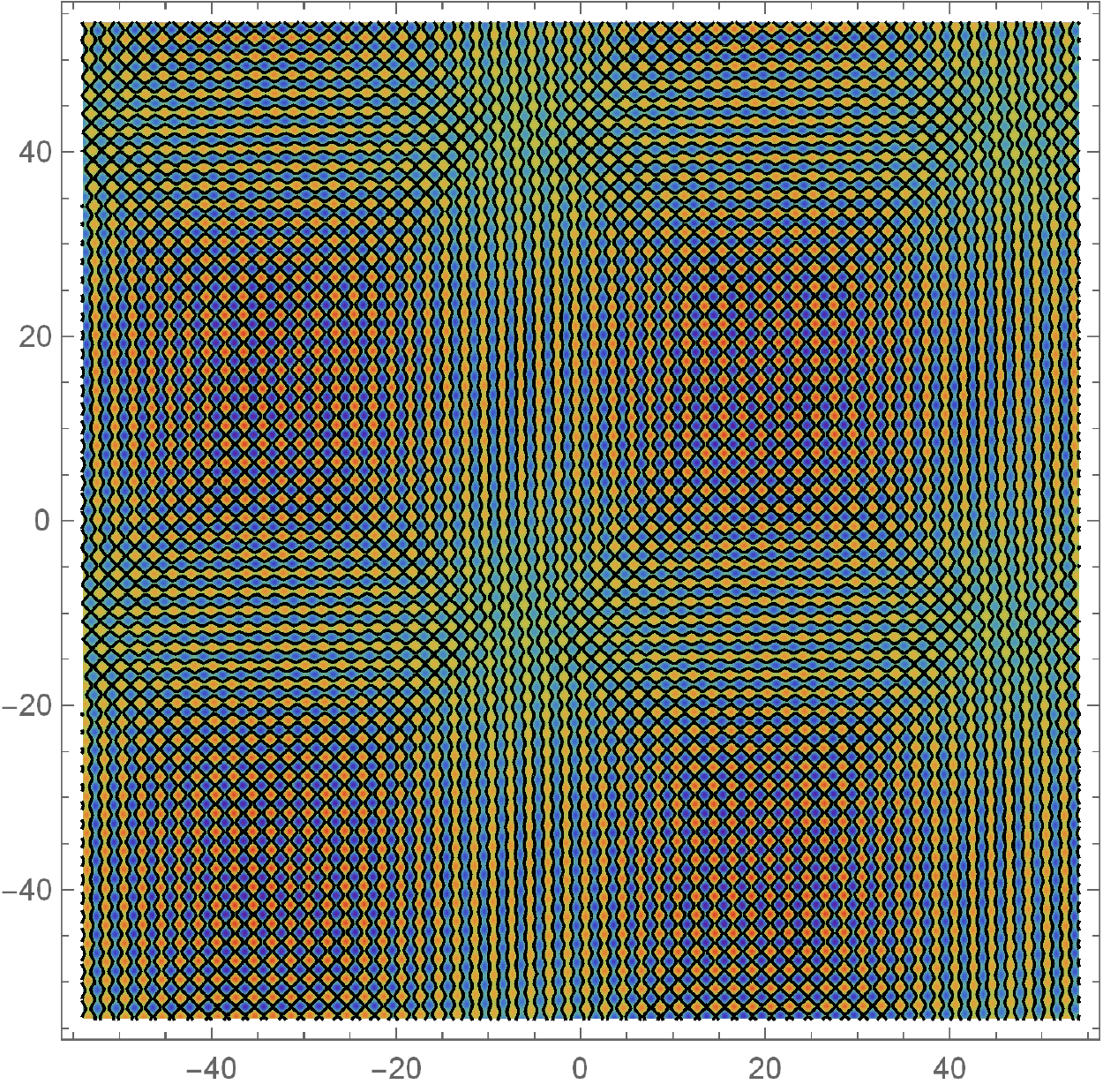}
	\end{center}
	\caption{Plot of a (rescaled) Cilleruelo type eigenfunction, for $m = 54^2 + 1$ and $r_2(m) = 8$. Nodal lines are in black.}
	\label{pic3}
\end{figure}


The field corresponding to the Cilleruelo measure is called the \emph{Cilleruelo field}. It can be written explicitly as
\begin{equation}
\label{cilfie}
F(x)=\frac{\sqrt{2}}{2}[b_1\cos(x_1)+c_1\sin(x_1)+b_2\cos(x_2)+c_2\sin(x_2)],
\end{equation}
were $b_1,b_2,c_1,c_2\sim\N(0,1)$ and i.i.d.
Figure \ref{pic1}  gives  examples of the nodal lines for a Cilleruelo field.
\begin{figure}
	\begin{center}
		\includegraphics[width=0.45\textwidth]{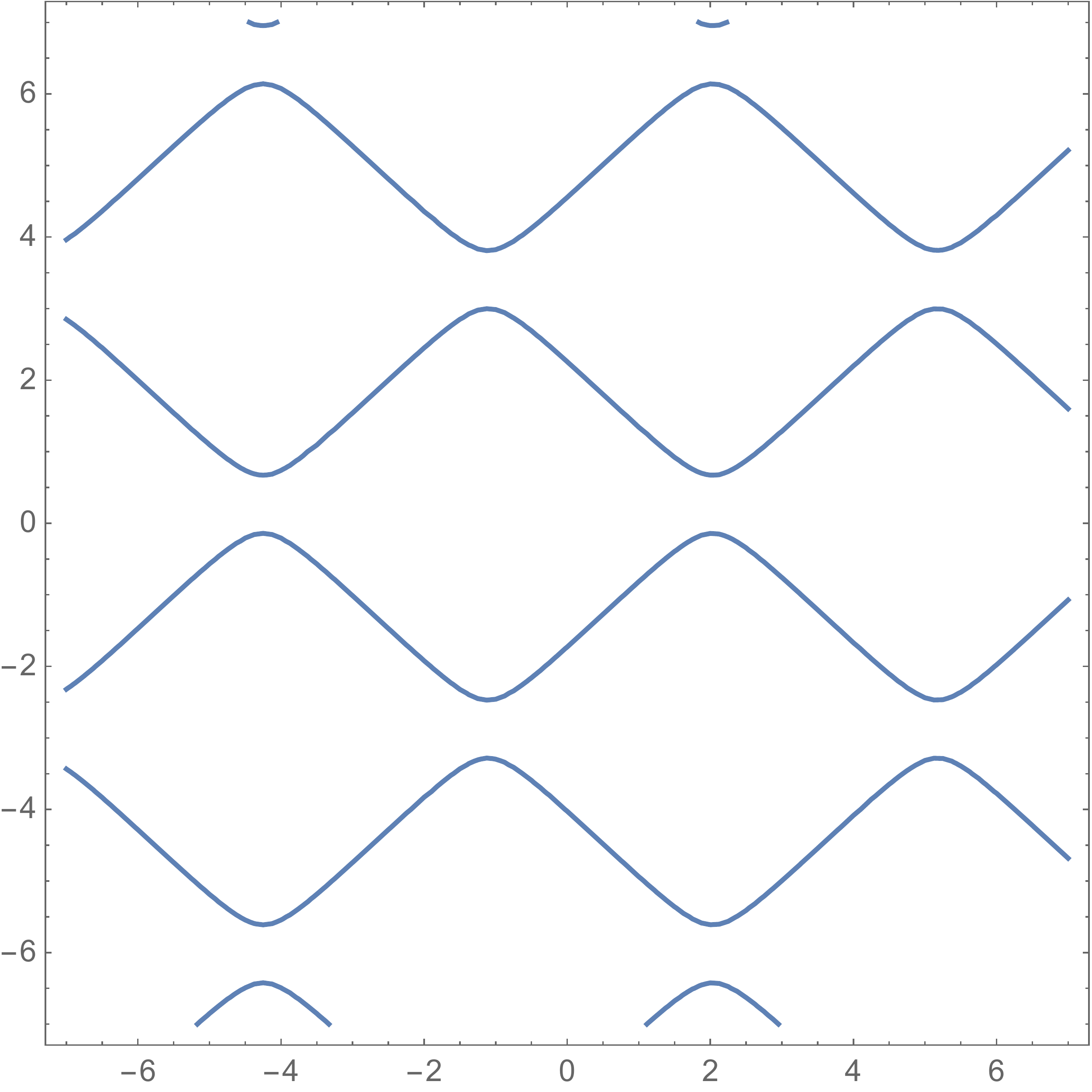}
		\includegraphics[width=0.45\textwidth]{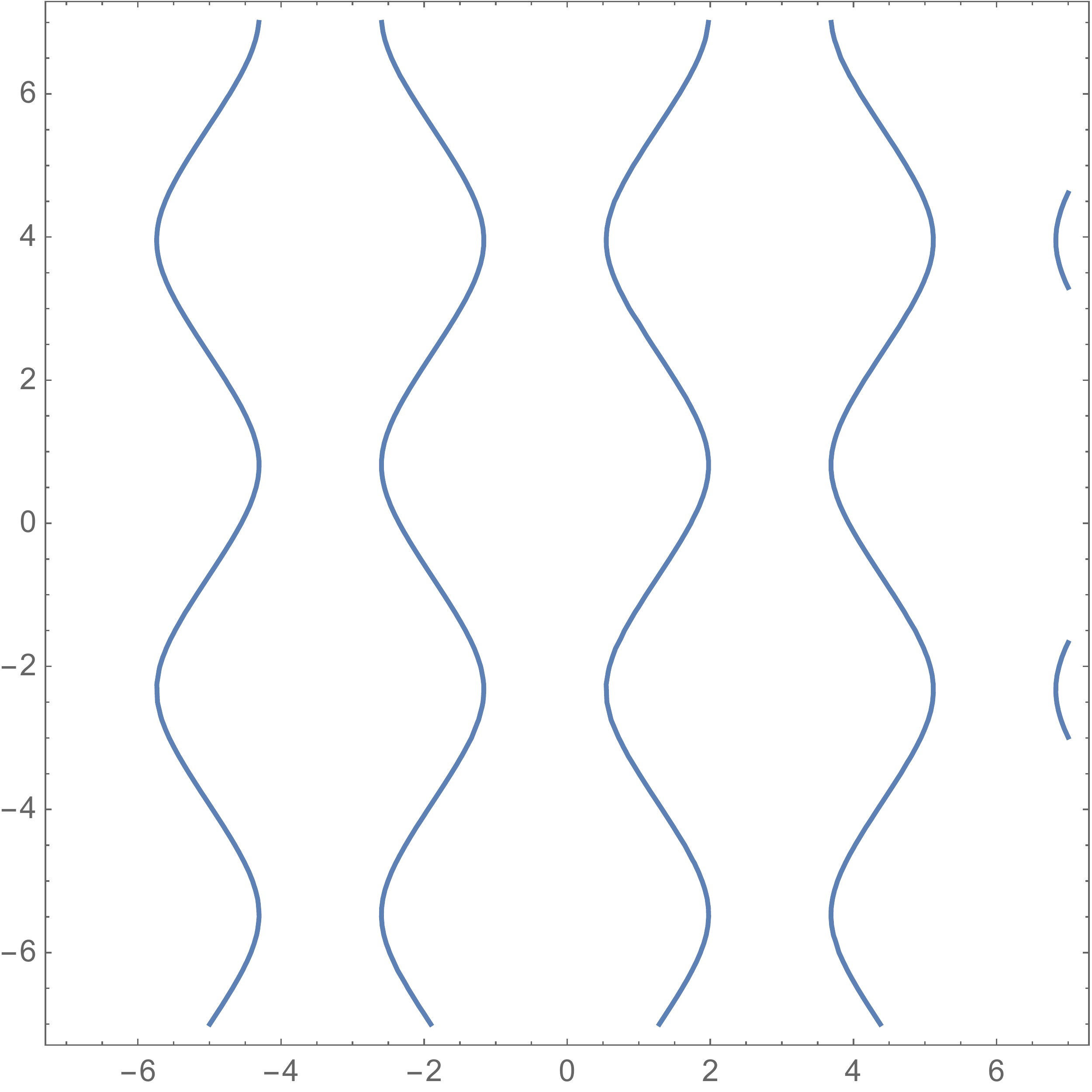}
	\end{center}
	\caption{Two samples of nodal lines of a (rescaled) Cilleruelo field.}
	\label{pic1}
\end{figure}

We observe that the nodal lines of Figure \ref{pic3}  resemble those of Figures \ref{pic1}. In Theorem \ref{discrete} we will give a precise meaning to this statement.

In this paper we are mostly interesting in understanding of behaviour of nodal sets as $m\to\infty$ in terms of the limiting spectral measure. In particular, we want to compare nodal sets with different energies. In this context it is natural to study rescaled fields defined by \eqref{rescaled_wave} since they all oscillate on the same scale and their spectral measures are all supported on the unit circle. Unless explicitly stated, all fields in this paper are assumed to be rescaled. It is also quite natural to extend by periodicity arithmetic waves from the (rescaled) torus to the entire plane.

\subsection{Nodal intersections}
For functions defined on a surface, a good indication about the geometry of the nodal set may be obtained by analysing its intersections with a fixed curve \cite{totzel,cantot,elhtot}. For (non-rescaled) arithmetic random waves $\Psi$, the expected nodal intersections number $\mathcal{Z}$ against a fixed smooth curve $\mathcal{C}\subset\mathbb{T}^2$ of length $L$ is \cite[Theorem 1.1]{rudwig}
\begin{equation}
\label{expectold}
\mathbb{E}[\mathcal{Z}]=\sqrt{2m}L.
\end{equation}
It is important to note that although the arithmetic random wave is not anisotropic, but this expectation is. The reason is that the field is stationary, and this is enough for such expectations to depend on the length of the curve only, but not on its shape or orientation.

Rudnick-Wigman \cite{rudwig} and subsequently Rossi-Wigman \cite{roswig} investigated the variance and distribution of $\mathcal{Z}$ against reference curves of {\em nowhere zero curvature}. The precise asymptotic behaviour of the variance is remarkably non-universal: it depends both on $\mathcal{C}$ and on the limiting spectral measure of the random waves.

Now take the intersection of the nodal line of $\Psi$ with a {\em fixed} straight line segment (the other extreme of nowhere zero curvature)
\begin{equation*}
\mathcal{C}: \gamma(t)=t(\cos(u),\sin(u)),
\qquad\qquad
0\leq t\leq L.
\end{equation*}
The expected intersection number \eqref{expectold} remarkably does not depend on the direction of the straight line. In \cite{maff2d} Maffucci bounded the nodal intersection variance with the same order of magnitude as the leading term in the work \cite{rudwig} of Rudnick-Wigman for the case of nowhere zero curvature curves. Here the problem is intimately related with the following so far unsolved question in analytic number theory. Consider a circle of radius $\sqrt{m}$: is it true that on every arc of length $(\sqrt{m})^{1/2}$ there are $O(1)$ lattice points as $m\to\infty$? One has the upper bound \cite{harwri}
\begin{equation}
\label{Nsmall}
r_2(m)=O(m^\epsilon)
\qquad\qquad\forall\epsilon>0.
\end{equation}
Jarnik \cite{jarnik} showed that on arcs of length $<(\sqrt{m})^{1/3}$ there are at most $2$ lattice points. Further progress is due to Cilleruelo-Granville \cite{cilgr2, cilgr1}.

In the present work we focus on the intersections of the nodal set for rescaled waves $\Psi$ against a straight line which has a fixed direction but it length varies.  Fix $0\leq u\leq\pi/4$ and take a sequence of straight line segments $\{\mathcal{C}_L\}_L$, $\mathcal{C}_L\subset\mathbb{T}^2$ along this direction,
\begin{equation}
\label{C2}
\mathcal{C}_L: \gamma(t)=t(\cos(u),\sin(u)),
\qquad\qquad\quad
0\leq t\leq L.
\end{equation}
We may assume $0\leq u\leq\pi/4$ thanks to the spectral measure symmetries. The nodal intersections $\mathcal{Z}=\mathcal{Z}(m,L)$ are the zeroes of the process $\psi:=\Psi(\gamma)$. How does $\Z$ depend on the direction of the line? We start with the asymptotic behaviour of the variance on {\em microscopic scales}, i.e. scales that are smaller than the wavelength.
\begin{prop}
	\label{2fmprop}
Let  $\{m_k\}_k$ be a sequence of energies such that the (rescaled) spectral measures converge to the limiting measure $\nu$. Let $C$ be an interval of length $L=L(m)$ in the direction $u$ and $\mathcal{Z}$ be the corresponding number of zeroes. Then, as $L\to 0$, we have
\begin{equation}
	\label{2fmasy}
	\E{\Z(\Z-1)}=L^3\cdot\frac{\sqrt{2}\pi^2}{24}(1+\widehat{\nu}(4)\cos(4u))+O(L^5),
\end{equation}
where $\widehat{\nu}(4)$ is the fourth Fourier coefficient of $\nu$.
\end{prop}

Unlike the expected number of zeroes, the leading coefficient in \eqref{2fmasy} depends on both the angle of the straight line and on the limiting spectral measure $\nu$. We also have the universal  bounds on the leading coefficient:
\begin{equation*}
0\leq\frac{\sqrt{2}\pi^2}{24}\cdot(1+\widehat{\nu}(4)\cos(4u))\leq\frac{\sqrt{2}\pi^2}{12}.
\end{equation*}
The leading coefficient may vanish in two cases: 
$\widehat{\nu}(4)=1$ (Cilleruelo measure) combined with $\cos(4u)=-1$ (diagonal line segments), and $\widehat{\nu}(4)=-1$ (tilted Cilleruelo measure) combined with $\cos(4u)=1$ (horizontal and vertical line segments). In these cases, the second factorial moment obeys the following smaller order asymptotic.

\begin{prop}
	\label{finerprop}
	Let $\{m_k\}_k$ and the corresponding $\mathcal{Z}$ be as in Proposition \ref{2fmprop}. We additionally assume that $\widehat{\nu}(4)\cos(4u)=-1$. Then the second factorial moment of zeroes has the asymptotic
	\begin{equation}
	\label{finer2}
	\E{\Z(\Z-1)}
	=L^5\cdot\frac{\sqrt{2}\pi^4}{450}
	+O(L^7), \quad L\to 0.
	\end{equation}
\end{prop}

Propositions \ref{2fmprop} and \ref{finerprop} 
will be proven in Section \ref{2fm}.

\subsection{Large deviations}
Our next results concern {\bf large deviations} of the zeroes along a straight line.  We are interested in the probability that a large gap occurs between two consecutive zeroes, called the {\bf persistence} probability. This is relevant at an {\em intermediate scale}, as we shall now see. 
We obtain an upper or a lower bound for the persistence in certain scenarios, depending on the relation between the spectral measure $\nu_m$ and the direction of the straight line. 

\begin{prop}
	\label{perlb}
	For fixed $m$ suppose $\nu_m$ has a point mass at $(\cos u,\sin u)$. Then one has
	\begin{equation}
	\label{noi}
	\p(\Z=0)\geq c>0
	\end{equation}
	independent of $L$. 
	
	Moreover, take a subsequence $\{m_k\}$ s.t. $\nu_{m_k}$ has a point mass at $(\cos u,\sin u)$ for all sufficiently big $m$, and $r_2(m_k)\to\infty$. Then for every $\epsilon>0$  there are constants $C_1$ and $C_2$ which are independent of $L$ such that for all sufficiently large $r_2(m)$
	\begin{equation}
	\label{noi2}
	-\log\p(\Z=0)<C_1 r_2(m)< C_2 m^\epsilon.
	\end{equation}
\end{prop}
To be more precise, we can take $C_1$ to be any constant such that $0<C_1<1/2$.


\begin{prop}
	\label{perub}
	For fixed $m$ suppose $\nu_m$ does not have a point mass at $(\cos u,\sin u)$. Then we have the upper bound
	\begin{equation}
	\log\p(\Z=0)=O(-L^2) \quad \text{ as } L\to \infty .
	\end{equation}
\end{prop}

Propositions \ref{perlb} and \ref{perub} will be proven in Section \ref{per}. 

Note that these propositions allow to  detect atoms of $\nu_m$ by considering the persistence probability in a given direction. To prove these propositions, we will make considerations about {\em spectral gaps} -- see Section \ref{per} for details.

When the limiting spectral measure has a nice form, we obtain sharp bounds for the persistence probability at the intermediate scale. The \mbox{weak-*} partial limits of $\{\nu_{m}\}$ (`attainable measures') were partially classified in \cite{krkuwi, kurwig}. Among other results, one has \cite{kurwig} the following: for every $\theta\in[0,\pi/4]$, the measure $\sigma_\theta$ that is uniform on the union of four arcs
\begin{equation*}
\{z\in\mathbb{S}^1 : \arg(z)\in\cup_{j=0}^{3}[j\cdot\pi/2-\theta,j\cdot\pi/2+\theta]\}
\end{equation*}
is attainable. In particular, when $\theta=0$ and $\theta=\pi/4$ we recover the Cilleruelo and Lebesgue measures respectively. We are now ready to state our next result.

\begin{prop}
	\label{persb}
	Fix $0\leq u\leq\pi/4$ and suppose that the  spectral measure of $\Psi$ is $\sigma_\theta$ with $\theta\neq 0$. Then we have
	\begin{equation}
	\label{tr1}
	\log\p(\Z=0)\asymp -L
	\qquad\qquad\text{ if }
	u\leq\theta
	\end{equation}
	and
	\begin{equation}
	\label{tr2}
	\log\p(\Z=0)\asymp -L^2
	\qquad\qquad\text{ if }
	\theta<u.
	\end{equation}
\end{prop}
Proposition \ref{persb} will be proven in Section \ref{per}. We will show how the phase transition between \eqref{tr1} and \eqref{tr2} comes from the existence of a spectral gap.

We now turn our attention to the Cilleruelo field $F$ \eqref{cilfie}, a case excluded from Proposition \ref{persb}. For $0\leq u\leq\pi/4$ consider the restriction
\begin{equation}
\label{rcf}
f_u(t)=F(t\cos(u),t\sin(u)),
\qquad\qquad
t\in[0,L],
\end{equation}
where $L$ is allowed to vary. Denote $\mathcal{Z}_f$ the zeroes of this process. Due to the normalisation \eqref{expectold} now reads
\begin{equation}
\label{expect}
	\E{\Z_f}=\frac{L}{\pi\sqrt{2}}.
\end{equation}
Since \eqref{rcf} is a stationary process, the expected number of zeroes is independent of the direction $u$. However, the process itself is very far from being rotation invariant. We then seek observables that depend on the direction $u$. We start by formulating several precise statements for the persistence probability of \eqref{rcf} at the intermediate scale, depending on $u$.
\begin{prop}
	\label{percill}
	There exist positive constants $C_n$ such that for the restricted Cilleruelo field $f_u$, $0\leq u\leq\pi/4$, we have
	\newcounter{pippo}
	\setcounter{pippo}{\value{equation}}
	\renewcommand{\theequation}{\roman{equation}}
	\setcounter{equation}{0}
	\begin{numcases}{\p(\Z_f=0)}
	=1-\frac{\sqrt{2}}{2}
	&
	$u=0,\ L\geq 2\pi$; \label{i}
	\\
	\geq C_0>0
	&
	$Lu\shortarrow{1}C_1<\pi/2$;
	\label{ii}
	\\
	\geq C_2\exp(-\frac{C_3}{(Lu-\pi/2)^2})
	&
	$Lu\shortarrow{1}\pi/2$;
	\label{iii}
	\\
	\leq\exp(-C_4L^2\sin(u)^2)
	&
	$u\text{ fixed,} \ 0<u<\pi/4$;
	\label{iv}
	\\
	\geq C_5(\epsilon)>0
	&
	$L\sin(u)<\pi-\epsilon$;
	\label{v}
	\\
	\leq C_6(\epsilon)<1
	&
	$L\sin(u)>\pi/2+\epsilon$;
	\label{vi}
	\\
	=0
	&
	$u>0, L\sin(u)\geq 2\pi$.
	\label{vii}
	%
	\end{numcases}
	\setcounter{equation}{\value{pippo}}
\end{prop}
The proof of Proposition \ref{percill} can be found in Section \ref{per}.

\subsection{Cilleruelo type fields}
We would like to return back to Figure \ref{pic3} which shows a sample of a field whose spectral measure is close to the Cilleruelo spectral measure. It is easy to see that on small intermediate scales this field looks like the Cilleruelo field. In this section we give a rigorous quantitative statement about this `similarity'.


Let $G$ be a stationary Gaussian field such that its spectral measure is purely atomic with all atoms having the same mass\footnote{This assumption is not crucial, but it makes some computations a bit simpler and it holds in the context of arithmetic waves.},  symmetric w.r.t. rotations by $\pi/2$ around the origin and supported on $\mathbb{S}^1$. Let $N=4M$ be the number of atoms and that are at points $y_j=(\cos(\varphi_j),\sin(\varphi_j))$,  $j=1,\dots,N$. We assume that $-\epsilon<\varphi_j<\epsilon$ for each $j=1,\dots,M$ for some fixed $\epsilon>0$. Then $G$ will be completely determined, via the symmetries, once we fix the $M$ atoms $\epsilon$-close to the point $(1,0)$. We may write explicitly
\begin{equation}
\label{G}
G(x)=\frac{\sqrt{2}}{2}\left[\sum_{j=1}^{2M}b_j\cos(\langle x,y_j\rangle)+\sum_{j=1}^{2M}c_j\sin(\langle x,y_j\rangle)\right],
\end{equation}
where $b_j,c_j\sim\N(0,2/N)$ for $j=1,\dots,2M$ and i.i.d. The  main goal is to show that such fields on intermediate scales look like the Cilleruelo field. To be more precise, we establish a natural coupling between such a field and the Cilleruelo field in such a way that the coupled field are close to each other with large probability.
\begin{thm}
\label{discrete}
Let $F$ be the Cilleruelo field and $G$ be a field  as above. Then there is an absolute constant $c$, and a coupling of the fields such that for all $R\ge 2c$
\[
\P\brb{\|F-G\|_{C(B_R)}\ge 2\epsilon R\log R}\le \exp(-\log^2 R/c)
\] 
and
\[
\P\brb{\|F-G\|_{C(B_R)}\ge 2\epsilon R^2 }\le \exp(-R^2/c).
\] 	
\end{thm}

We would like to point out that this is very close to the best possible result since, as we will see, the difference between covariance kernels of $F$ and $G$ inside $B_R$ is of order $\epsilon R$.

Theorem \ref{discrete} will be proven in Section \ref{cou}. 

\begin{rem}
We state this theorem for coupling in the disc centred at the origin, but by stationarity, the same holds for any disc. Note, that the same sample Cilleruelo-type field will be coupled to different samples of the Cilleruelo field in different discs. See Figure \ref{pic5}.
\end{rem}

\begin{rem} This is a particular example of coupling of fields with close spectral measures. The general coupling result will appear elsewhere.
\end{rem}

\begin{rem}
Similar result with essentially the same proof holds for $\|F-G\|_{C^k(B_R)}$ for any integer $k$. With a bit more work, it can be shown that it is also true for non-integer $k$.
\end{rem}

\begin{rem}
Note that Theorem \ref{discrete} states that the coupled fields are close. In general, this does not mean that their nodal lines are close. If two functions are close, then in order for their nodal lines to be close, we also need the nodal lines to be `stable', that is, we need that the gradient can not be too small on the nodal line.  Figure \ref{pic5} gives an example of a `stable' and `unstable' couplings. 
This could be quantified and the probability of the `unstable' nodal lines can be estimated. Since this is not important for our considerations, we are not providing the details. We refer interested readers to  \cite[Lemmas 8 and 9]{BeWi} as well as to \cite{NaSo15,Sodin}. 
\end{rem}

\begin{figure}
	\begin{center}
		\includegraphics[width=0.45\textwidth]{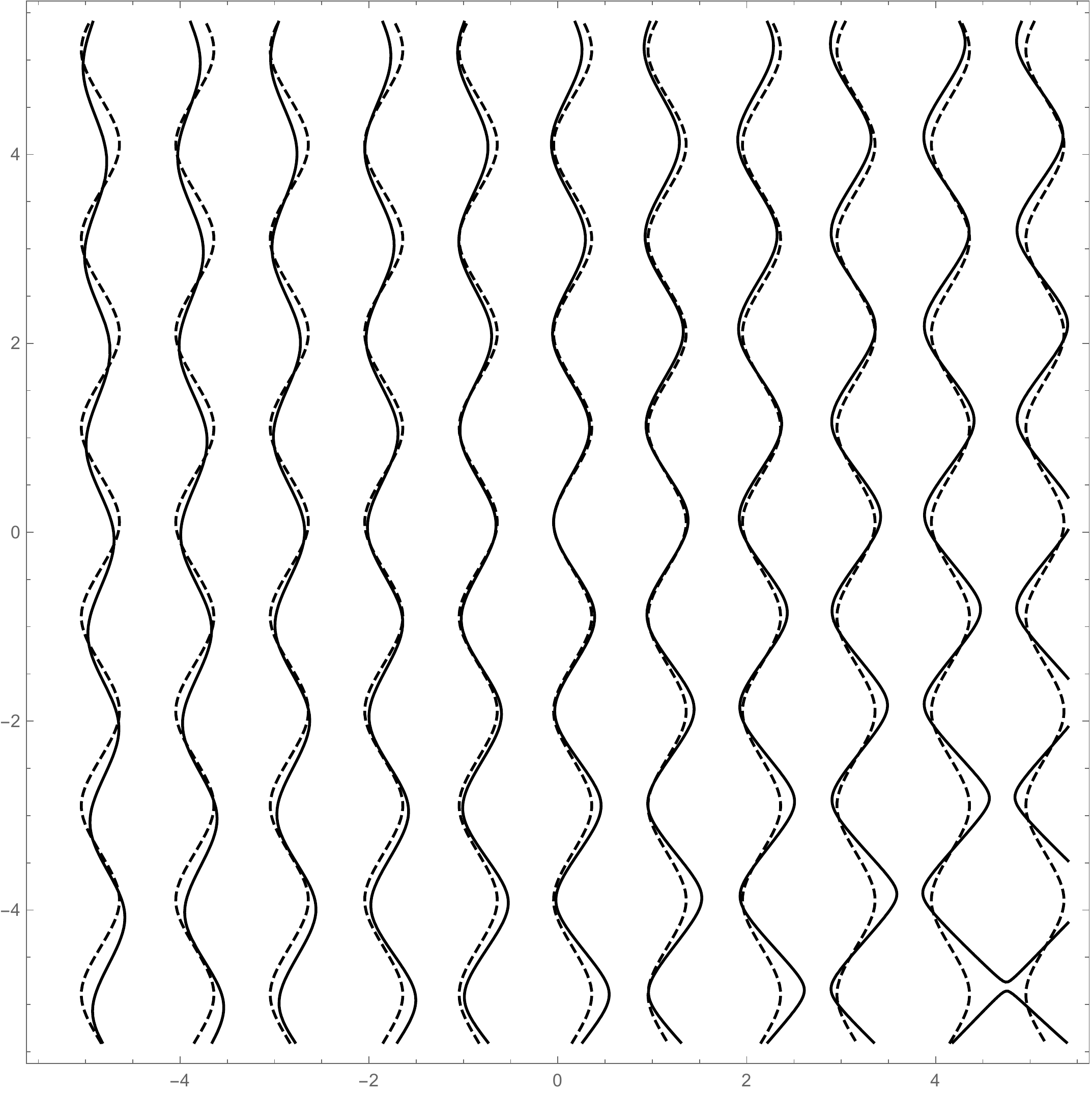} \
		\includegraphics[width=0.45\textwidth]{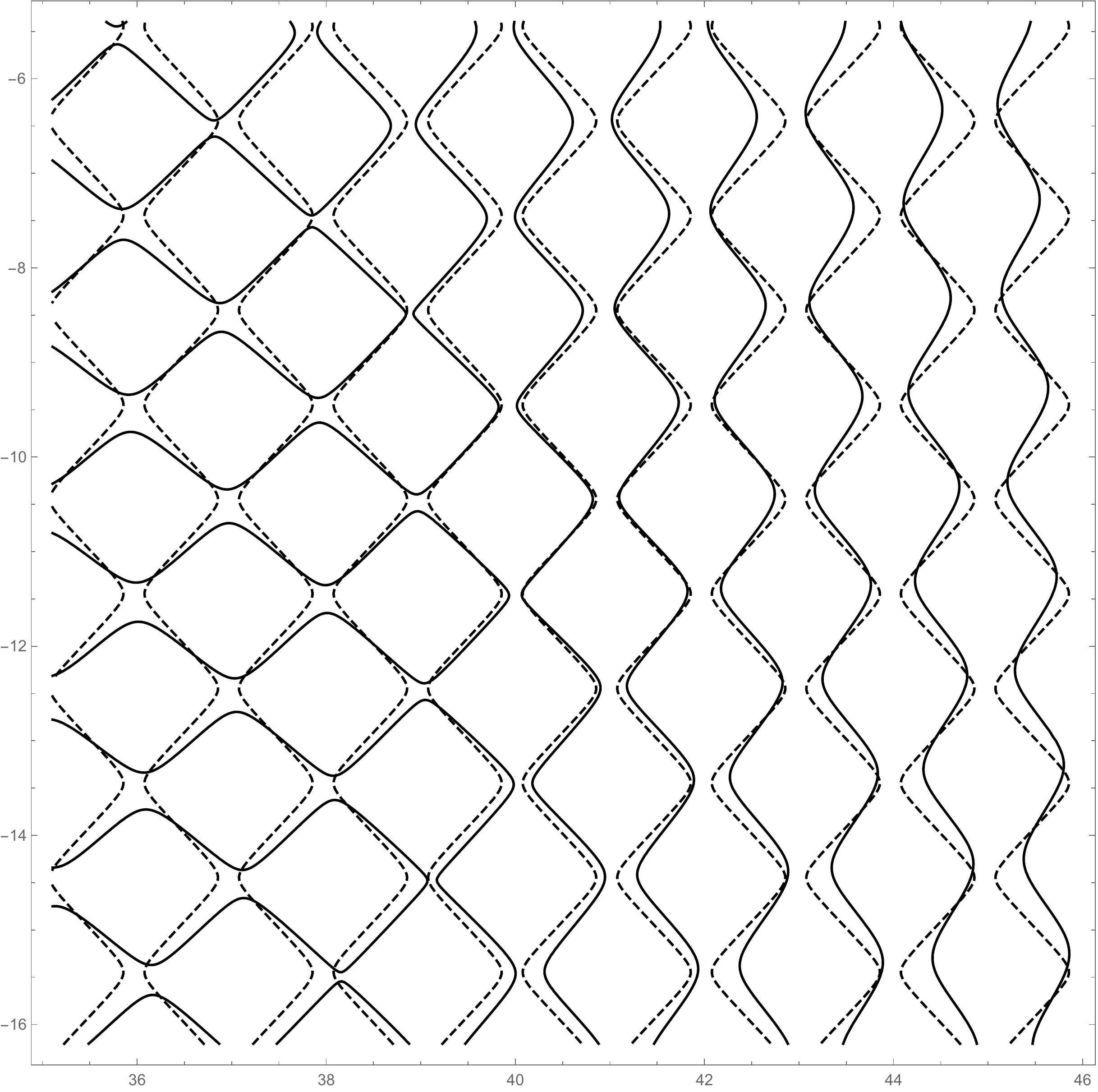}
	\end{center}
	\caption{Both pictures show the coupling of the same sample of a Cilleruelo-type function as in Figure \ref{pic3} but in two different boxes. Solid line shows the nodal lines of the Cilleruelo-type field and dashed are the nodal lines of the coupled Cilleruelo fields. The left is an example of a `stable' nodal line, the right is an example of an `unstable' one.}
	\label{pic5}
\end{figure}

For small $L\epsilon$ Theorem \ref{discrete} allows us to control the persistence probability of $g$ as in the following result.
\begin{prop}
	\label{perfg}
	Fix $0\leq u\leq\pi/4$, $L\geq 2\pi$, and $\epsilon>0$. Let $F$ be a Cilleruelo field and $G$ be of Cilleruelo type so that the spectral measures of $F,G$ are $\epsilon$-close. Denote $\mathcal{Z}_f,\mathcal{Z}_g$ the number of nodal intersections of respectively $F,G$ against the straight line $\mathcal{C}$ of direction $u$ and length $L=O(\epsilon^{\alpha-1})$ with some $\alpha>0$. Then we have 
\begin{equation}
	\label{perg1}
\begin{aligned}
\mathbb{P}(\mathcal{Z}_g=0)&=O(\epsilon^\alpha\log \epsilon)& u\ne 0
\\
\mathbb{P}(\mathcal{Z}_g=0)&
\ge 1-\frac{\sqrt{2}}{2}+O(\epsilon^\alpha\log \epsilon)& u= 0.
\end{aligned}
\end{equation}
\end{prop}
The proof of Proposition \ref{perfg} may be found in Section \ref{cou}. Theorem \ref{discrete} and Proposition \ref{perfg} give indications on the closeness of the nodal lines of $F$ and $G$ when the spectral measures are close. At an appropriate scale, the nodal lines of a Cilleruelo type field `look like' those of a Cilleruelo field -- cfr. Figure \ref{pic5}.

The rest of this paper is organised as follows. In Section \ref{2fm}, we establish Propositions \ref{2fmprop} and \ref{finerprop}. 
In Section \ref{per}, we cover relevant background on persistence probability and spectral gaps, and prove Propositions \ref{perlb}, \ref{perub}, \ref{persb}, and \ref{percill}. In Section \ref{cou}, we establish the coupling results Theorem \ref{discrete} and Proposition \ref{perfg}.

\subsection*{Acknowledgements}
The authors would like to thank Naomi Feldheim and Michael McAuley for helpful discussions. D.B. and R.M. were supported by the Engineering \& Physical Sciences Research Council (EPSRC) Fellowship EP/M002896/1 held by D.B.

\section{Second factorial moment asymptotic}
\label{2fm}
The aim of this section is to prove Propositions \ref{2fmprop} and \ref{finerprop}. 
We will need background on Kac-Rice formulas for processes, a standard tool for computing moments for the number of zeroes.
\subsection{Kac-Rice formulas}
The restriction of the rescaled arithmetic random wave $\Psi_m$ \eqref{rescaled_wave} to a straight line in the direction $u$ defines a centred Gaussian stationary process $\psi=:[0,L]\to\mathbb{R}$,
\begin{equation}
\label{psi}
\psi_u(t)=
\frac{1}{\sqrt{r_2(m)}}
\sum_{\lambda\in\text{supp}(\nu_m)}
a_{\lambda}
e^{2\pi it\langle\lambda,(\cos(u),\sin(u))\rangle}.
\end{equation}
 
Its covariance function is given by the expression
\begin{equation}
\label{rd}
\kappa(t)=\E{\psi(t_1)\psi(t_2)}
=
\frac{1}{r_2(m)}\sum_{\lambda\in\text{supp}(\nu_m)} e^{2\pi it\langle\lambda,(\cos(u),\sin(u))\rangle}
\end{equation}
where $t=t_1-t_2$ with slight abuse of notation. For a process $p$ satisfying certain conditions, moments of the number of zeroes
\begin{equation*}
\xi(p,T):=|\{t\in T : p(t)=0\}|
\end{equation*}
may be computed via {\bf Kac-Rice formulas} \cite{azawsc, cralea, adltay}. Let $p:I\to\mathbb{R}$ be a (a.s. $C^1$-smooth, say) Gaussian process on an interval $I\subseteq\mathbb{R}$. Define the first and second {\em intensities}, also called respectively \textbf{zero density} function $K_1: I\to\mathbb{R}$,
\begin{equation}
\label{K1gen}
K_1(t)=\phi_{p(t)}(0)\cdot\mathbb{E}[|p'(t)|\ \big| \ p(t)=0],
\end{equation}
and \textbf{2-point correlation} function $K_2: I\times I\to\mathbb{R}$,
\begin{equation}
\label{K2tgen}
K_2(t_1,t_2)=\phi_{p(t_1),p(t_2)}(0,0)\cdot\mathbb{E}[|p'(t_1)|\cdot|p'(t_2)|\ \big| \ p(t_1)=p(t_2)=0],
\end{equation}
for $t_1\neq t_2$. If $p$ is a stationary process, \eqref{K1gen} and \eqref{K2tgen} simplify to $K_1(t)\equiv K_1$ and
\begin{equation*}
K_2(t)=\phi_{p(0),p(t)}(0,0)\cdot\mathbb{E}[|p'(0)|\cdot|p'(t)|\ \big| \  p(0)=p(t)=0]
\end{equation*}
respectively.
\begin{thm}[{\cite[Theorem 3.2]{azawsc}}]
\label{krproc}
Let $p$ be a real-valued Gaussian process defined on an interval $I$ and having $C^1$ paths. Denote $\xi$ the number of zeros of $p$ on $I$. Then the expected number of zeroes is given by
\begin{equation}
\label{kr1}
\mathbb{E}[\xi]=\int_{I}K_1(t)dt.
\end{equation}
Assume further that for every $t_1\neq t_2$ the joint distribution of the random vector $(p(t_1),p(t_2))\in\mathbb{R}^2$ is non-degenerate. Then one has
\begin{equation}
\label{kr2}
\mathbb{E}[\xi(\xi-1)]=
\int_{I^2}K_2(t_1,t_2)dt_1dt_2.
\end{equation}
\end{thm}
Letting $\Z=\Z(m,u)$ be the number of zeroes of the process $\psi$, Rudnick and Wigman \cite{rudwig} computed via \eqref{kr1}  $\mathbb{E}[\mathcal{Z}]=\sqrt{2}L$ \eqref{expectold}, independent of $u$. For the second factorial moment, the non-degeneracy assumption of Theorem \eqref{krproc} is unfortunately far from being satisfied for the process $\psi$ \cite{rudwig, maff2d}. However, on {\em small scales} (e.g. in the limit $L\to 0$) the non-degeneracy assumption does hold \cite[Lemma 4.3]{rudwig}. To establish the asymptotic of the second factorial moment, we will need a preliminary lemma. Recall the notation \eqref{lp} for the lattice point set $\Lambda=\Lambda_m=\sqrt{m}\,\text{supp}(\nu_m)$.
\begin{lemma}
\label{p}
For fixed $m$ and $\alpha=(\cos(u),\sin(u))\in\mathbb{R}^2$ one has
\begin{equation}
\label{p2}
\frac{1}{|\Lambda|}\sum_{\lambda\in\Lambda}\langle\lambda,\alpha\rangle^2=\frac{m}{2},
\end{equation}
\begin{equation}
\label{p4}
\frac{1}{|\Lambda|}\sum_{\lambda\in\Lambda}\langle\lambda,\alpha\rangle^4=\frac{m^2}{8}(3+\widehat{\nu}(4)\cos(4u)),
\end{equation}
and
\begin{equation}
\label{p6}
\frac{1}{|\Lambda|}\sum_{\lambda\in\Lambda}\langle\lambda,\alpha\rangle^6=\frac{m^3}{16}(5+3\widehat{\nu}(4)\cos(4u)).
\end{equation}
\end{lemma}
\begin{proof}
The statement \eqref{p2} was proven in \cite[Lemma 2.2]{rudwi2}. To show \eqref{p4}, we begin by defining 
\begin{equation*}
A:=\frac{1}{m^2}\sum_{\Lambda}\lambda_1^4=\frac{1}{m^2}\sum_{\Lambda}\lambda_2^4,
\qquad\qquad
B:=\frac{1}{m^2}\sum_{\Lambda}\lambda_1^2\lambda_2^2,
\qquad\qquad
\lambda=(\lambda_1,\lambda_2)
\end{equation*}
and writing
\begin{equation}
\label{p4st1}
\sum_{\Lambda}\langle\lambda,\alpha\rangle^4=m^2\left[(\alpha_1^4+\alpha_2^4)A+6\alpha_1^2\alpha_2^2B\right],
\end{equation}
where the remaining summands cancel out in pairs by the symmetries of the set $\Lambda$. Since
\begin{equation*}
\widehat{\nu}(4)=\frac{1}{|\Lambda|}\sum_{\Lambda}\left(\frac{\lambda_1+i\lambda_2}{\sqrt{m}}\right)^4=\frac{1}{|\Lambda|}(2A-6B),
\end{equation*}
and since clearly $2(A+B)=|\Lambda|$, one obtains
\begin{equation}
\label{p4st2}
A=\frac{|\Lambda|}{8}(3+\widehat{\nu}(4)),
\qquad\qquad
B=\frac{|\Lambda|}{8}(1-\widehat{\nu}(4)).
\end{equation}
Replacing \eqref{p4st2} into \eqref{p4st1} we get
\begin{equation}
\sum_{\Lambda}\langle\lambda,\alpha\rangle^4=\frac{m^2|\Lambda|}{8}\left[3+\widehat{\nu}(4)\cdot(\alpha_1^4+\alpha_2^4-6\alpha_1^2\alpha_2^2)\right],
\end{equation}
where we used $|\alpha|=1$. We now obtain \eqref{p4} by observing that 
\[
\alpha_1^4+\alpha_2^4-6\alpha_1^2\alpha_2^2=\cos(4u)
\]
via the usual trigonometric identities.

We now outline the proof of \eqref{p6}, that is very similar to that of \eqref{p4}. We write
\begin{equation}
\label{p6st1}
\sum_{\Lambda}\langle\lambda,\alpha\rangle^6=m^3\left[(\alpha_1^6+\alpha_2^6)C+15(\alpha_1^4\alpha_2^2+\alpha_2^4\alpha_1^2)D\right],
\end{equation}
where
\begin{equation*}
C:=\frac{1}{m^3}\sum_{\Lambda}\lambda_1^6,
\qquad\qquad
D:=\frac{1}{m^3}\sum_{\Lambda}\lambda_1^4\lambda_2^2.
\end{equation*}
We clearly have $C+D=A={|\Lambda|}(3+\widehat{\nu}(4))/8$ and $2(C+3D)=|\Lambda|$, whence
\begin{equation}
\label{p6st2}
C=\frac{|\Lambda|}{16}(5+3\widehat{\nu}(4)),
\qquad\qquad
D=\frac{|\Lambda|}{16}(1-\widehat{\nu}(4)).
\end{equation}
Substituting \eqref{p6st2} into \eqref{p6st1} yields
\begin{equation}
\label{p6st3}
\sum_{\Lambda}\langle\lambda,\alpha\rangle^6=\frac{m^3|\Lambda|}{16}\left[5(\alpha_1^6+\alpha_2^6+3\alpha_1^2\alpha_2^2)+3\widehat{\nu}(4)\cdot(\alpha_1^6+\alpha_2^6-5\alpha_1^2\alpha_2^2)\right],
\end{equation}
where we also noted that $\alpha_1^4\alpha_2^2+\alpha_2^4\alpha_1^2=\alpha_1^2\alpha_2^2$ since $|\alpha|=1$. One has $\alpha_1^6+\alpha_2^6+3\alpha_1^2\alpha_2^2=1$ and $\alpha_1^6+\alpha_2^6-5\alpha_1^2\alpha_2^2=\cos(4u)$. We replace these into \eqref{p6st3} to establish \eqref{p6}.
\end{proof}

\subsection{The proofs of Propositions \ref{2fmprop} and \ref{finerprop}}
We now formulate the asymptotic of the two-point correlation function \eqref{K2tgen} of $\psi$ \eqref{psi} on small scales.
\begin{prop}
\label{K2asyprop}
As $L\to 0$, the two-point correlation function has the form
\begin{equation}
\label{K2asy}
K_{2;m}(t_1,t_2)=\frac{\sqrt{2}\pi^2}{8}|t_1-t_2|\cdot(1+\widehat{\nu}(4)\cos(4u))+O(|t_2-t_1|^3).
\end{equation}
\end{prop}
The proof of Proposition \ref{K2asyprop} is deferred to the end of this section.

\begin{proof}[Proof of Proposition \ref{2fmprop} assuming Proposition \ref{K2asyprop}]
By \cite[Lemma 4.3]{rudwig}, the Kac-Rice formula \eqref{kr2} holds for $L\to 0$. To establish \eqref{2fmasy}, we replace the asymptotic \eqref{K2asy} into \eqref{kr2} and compute the integrals
\begin{equation*}
\iint_{[0,L]^2}|t_1-t_2|dt_1dt_2=\frac{L^3}{3}
\quad
\text{and}
\quad
\iint_{[0,L]^2}|t_1-t_2|^3dt_1dt_2=\frac{L^5}{10}.
\end{equation*}
\end{proof}

We have the following bounds for the leading coefficient in \eqref{K2asy}:
\begin{equation*}
0\leq\frac{\sqrt{2}\pi^2}{8}\cdot(1+\widehat{\nu}(4)\cos(4u))\leq\frac{\sqrt{2}\pi^2}{4}.
\end{equation*}
If $\widehat{\nu}(4)>0$, the maximum is reached when $u$ is a multiple of $\pi/2$ and the minimum is reached when $u$ is an odd multiple of $\pi/4$; whereas if $\widehat{\nu}(4)<0$, the opposite is true. If $\widehat{\nu}(4)=0$, that is to say, if the lattice points $\Lambda$ equidistribute in the limit, then the leading coefficient does not depend on the angle $u$.

On the other hand, if $u=\pi/8$ or an odd multiple, then $\cos(4u)=0$ hence \eqref{K2asy} is independent of the distribution of the lattice points on $\sqrt{m}\mathbb{S}^1$. This is tantamount to
\begin{equation*}
\frac{1}{|\Lambda|}\sum_{\Lambda}\langle\lambda,\alpha\rangle^4=\frac{3m^2}{8}
\end{equation*} 
i.e. the vanishing of 
\begin{equation*}
\alpha_1^4+\alpha_2^4-6\alpha_1^2\alpha_2^2.
\end{equation*}
The leading coefficient of \eqref{K2asy} may vanish in the two cases: 
$\widehat{\nu}(4)=1$ (Cilleruelo measure) combined with $\cos(4u)=-1$ (diagonal line segments), and $\widehat{\nu}(4)=-1$ (tilted Cilleruelo measure) combined with $\cos(4u)=1$ (horizontal and vertical line segments). In these cases, the two-point correlation function obeys the following smaller order asymptotic.

\begin{prop}
\label{finerK2}
Assume $L\to 0$ and $\widehat{\nu}(4)\cos(4u)=-1$. Then the two-point correlation function has the form
\begin{equation}
\label{finer}
K_{2;m}(t_1,t_2)
=\frac{\sqrt{2}\pi^4}{45}m^{5/2}|t_1-t_2|^3
+O(|t_2-t_1|^5).
\end{equation}
\end{prop}
Before proving Proposition \ref{finerK2}, we complete the proof of Proposition \ref{finerprop}.
\begin{proof}[Proof of Proposition \ref{finerprop} assuming Proposition \ref{finerK2}]
By \cite[Lemma 4.3]{rudwig}, the Kac-Rice formula \eqref{kr2} holds for $L\to 0$. To show \eqref{finer2}, we insert \eqref{finer} into \eqref{kr2} and compute
\begin{equation*}
\iint_{[0,L]^2}|t_1-t_2|^3dt_1dt_2=\frac{L^5}{10}
\quad
\text{and}
\quad
\iint_{[0,L]^2}|t_1-t_2|^5dt_1dt_2=\frac{L^7}{21}.
\end{equation*}
\end{proof}
To finish this section we prove Propositions \ref{K2asyprop} and \ref{finerK2}.
\begin{proof}[Proof of Proposition \ref{K2asyprop}]
The covariance function $\kappa$ \eqref{rd} satisfies $\kappa''(0)=-2\pi^2=:-M$. We suppress the dependency on $t_1,t_2,t=t_1-t_2$ to simplify notation. The two-point function admits the explicit expression \cite[Lemma 3.1]{rudwig}
\begin{equation}
\label{K2expl}
K_2=\frac{M(1-\kappa^2)-\kappa'^2}{\pi^2(1-\kappa^2)^{3/2}}(\sqrt{1-\rho^2}+\rho\arcsin\rho),
\end{equation}
where $\rho$ is the correlation between $\psi'(t_1)$ and $\psi'(t_2)$, and is given by
\begin{equation}
\label{rho}
\rho=\frac{\kappa''(1-\kappa^2)+\kappa\kappa'^2}{M(1-\kappa^2)-\kappa'^2}.
\end{equation}

Expanding exponent into power series and using  Lemma \ref{p} we obtain the following expansions as $t\to 0$
\begin{align*}
\kappa&\sim 1-\frac{Mt^2}{2}+\frac{M^2t^4}{6}\cdot\frac{1}{8}(3+\widehat{\nu}(4)\cos(4u))+O(t^6),
\\
\kappa'&\sim -Mt+\frac{M^2t^3}{12}(3+\widehat{\nu}(4)\cos(4u))+O(t^5),
\\
\kappa''&\sim -M+\frac{M^2t^2}{4}(3+\widehat{\nu}(4)\cos(4u))+O(t^4).
\end{align*}
Using these expansions we get
\begin{equation}
\label{ob1}
\frac{M(1-\kappa^2)-\kappa'^2}{\pi^2(1-\kappa^2)^{3/2}}
=
\frac{1}{8\pi^2}M^{3/2}t\cdot(1+\widehat{\nu}(4)\cos(4u))+O(t^3)
\end{equation}
and
\[
\rho=1+O(t^2).
\]
Combining these formulas we have
\[
K_2=\frac{1}{16\pi}M^{3/2}t\cdot(1+\widehat{\nu}(4)\cos(4u))+O(t^3).
\]
Recalling that $M=2\pi^2$ we complete the proof of the lemma.

\end{proof}

\begin{proof}[Proof of Proposition \ref{finerK2}]
We proceed as in the proof of Proposition \ref{K2asyprop}, computing the Taylor expansions to one order higher, e.g.,
\begin{align*}
\kappa&\sim 1-\frac{Mt^2}{2}+\frac{M^2t^4}{6}\cdot\frac{1}{8}(3+\widehat{\nu}(4)\cos(4u))-\frac{M^3t^6}{90}\cdot\frac{1}{16}(5+3\widehat{\nu}(4)\cos(4u)),
\\
\kappa'&\sim -Mt+\frac{M^2t^3}{12}(3+\widehat{\nu}(4)\cos(4u))-\frac{M^3t^5}{240}(5+3\widehat{\nu}(4)\cos(4u)),
\\
\kappa''&\sim -M+\frac{M^2t^2}{4}(3+\widehat{\nu}(4)\cos(4u))-\frac{M^3t^4}{48}(5+3\widehat{\nu}(4)\cos(4u))
\end{align*}
(via Lemma \ref{p}). As before, we plug this into the formula for $K_2$ and using that $\widehat{\nu}(4)\cos(4u)=-1$ we have
\[
K_{2;m}
=-\frac{\sqrt{2}\pi^4}{90}t^3\cdot\left[25+32\widehat{\nu}(4)\cos(4u)+5\widehat{\nu}(4)^2\cos(4u)^2\right]
\\+O(t^5 ).
\]
\end{proof}

\section{Persistence probability}
\label{per}
\subsection{Random toral waves}
The aim of this section is to prove Propositions \ref{perlb} and \ref{perub}. Recall that the restriction of a toral wave \eqref{arwd} to the straight line \eqref{C2} defines the process \eqref{psi} $\psi_u:[0,L]\to\mathbb{R}$,
\begin{equation*}
\psi_u(t)=
\frac{1}{\sqrt{r_2(m)}}
\sum_{\lambda\in\text{supp}(\nu_m)}
a_{\lambda}
e^{2\pi it\langle\lambda,(\cos(u),\sin(u))\rangle}.
\end{equation*}
The spectral measure of this process is the projection of the original spectral measure
\begin{equation*}
\rho_{u}=\frac{1}{r_2(m)}\sum_{\lambda\in\text{supp}(\nu_m)}\delta_{\langle\lambda,(\cos(u),\sin(u)\rangle}.
\end{equation*}

\begin{proof}[Proof of Proposition \ref{perlb}]
The main idea of the proof is very simple. Let us assume that  $\nu_m$ has a mass at $(\cos(u),\sin(u))=\tilde{\lambda}$, then the term corresponding to this point-mass (and its rotations by multiples of $\pi/2$) in the definition of $\Psi$ is a rotated Cilleruelo field. With positive probability its nodal set does not intersect a given line in this direction (see Figure \ref{pic1}). With positive probability, this term dominates the entire sum and the same is true for $\Psi$ and its restriction $\psi_u$.

To prove the first part of the proposition we note that $\psi_u(t)$ has no zeros if
\begin{equation*}
\sum_{\langle\lambda,(\cos(u),\sin(u))\rangle\neq 0}\sqrt{b_\lambda^2+c_\lambda^2}<|a_{\tilde{\lambda}}|,
\end{equation*}
where we recall $a_\lambda=b_\lambda+ic_\lambda$. For $|\Lambda_m|$ fixed, one immediately obtains \eqref{noi}.

To show \eqref{noi2}, for each $\lambda\neq\tilde{\lambda}$, consider the i.i.d. random variables $X_\lambda:=\sqrt{b_\lambda^2+c_\lambda^2}\sim\chi(2)$. By the CLT one has
\begin{equation*}
\p(\Z=0)\gtrsim\p(\sqrt{r_2(m)}Z_1+r_2(m)<|Z_2|),
\end{equation*}
where $Z_1,Z_2$ are i.i.d. standard Gaussians. It follows that
\[
\begin{aligned}
\p(\Z=0)&\gtrsim\p\br{|Z_1+\sqrt{r_2(m)}|<\frac{1}{\sqrt{r_2(m)}}}\cdot\p(|Z_2|>1)
\\
&\asymp\frac{e^{-r_2(m)/2}}{\sqrt{r_2(m)}},
\end{aligned}
\]
as claimed. The second inequality in \eqref{noi2} follows on recalling \eqref{Nsmall}.
\end{proof}

In the argument above we see that if the spectral measure $\rho_u$ has an atom at the origin, then $\psi_u$ contains a constant term which dominates the entire sum with positive probability. If it has no atom at the origin, then, in some sense, the  most important term is the one which corresponds to the atom which is closest to the origin. The precise formulation is fiven by the following result by Feldheim-Feldheim-Jaye-Nazarov-Nitzan \cite{ffjnn1}. Given a random process, we say that it has a {\bf spectral gap} if its spectral measure vanishes on an interval $(-a,a)$.
\begin{thm}[{\cite[Theorem 1]{ffjnn1}}]
\label{sgc}
A continuous Gaussian process on $[0,T]$ with spectral gap $(-a,a)$ has persistence probability
\begin{equation}
\p(\Z=0)\leq\exp(-ca^2T^2)
\end{equation}
for some absolute constant $c$.
\end{thm}
We are now in a position to prove Proposition \ref{perub}.
\begin{proof}[Proof of Proposition \ref{perub}]
Since $(\cos u,\sin u)$ is not in the support of $\nu_m$, the spectral measure of $\psi$ has the gap $(-A_m,A_m)$, where
\begin{equation*}
A_m:=\min_{\lambda\in\text{supp}(\nu_m)}\left\{\left|\left\langle\lambda,(\cos u,\sin u)\right\rangle\right|\right\}>0.
\end{equation*}
By Theorem \ref{sgc} it follows that
\begin{equation}
\p(\Z=0)\leq\exp(-\bar{c}A_m^2L^2),
\end{equation}
where $\bar{c}$ is a constant.
\end{proof}

\subsection{Sharp bounds and phase transition}
The goal of this section is to prove Proposition \ref{persb}. Recall that for a density one sequence of energies ${m_k}$ one has \eqref{weak}
\begin{equation*}
\nu_{m_k}\Rightarrow\frac{d\theta}{2\pi}.
\end{equation*}
In this case, the limiting spectral measure $\rho$ of $\psi=\Psi(\gamma)$ satisfies
\begin{equation*}
\rho([0,T])=\rho([-T,0])=\frac{1}{2}-\frac{1}{\pi}\arccos(T),
\end{equation*}
independent of the direction $u$. As mentioned in the Introduction, more generally for every $\theta\in[0,\pi/4]$ the limiting spectral measure $\sigma_{\theta}$ of $\Psi$ supported on
\[
\{z\in\mathbb{S}^1 : \arg(z)\in\cup_{j=0}^{3}[j\cdot\pi/2-\theta,j\cdot\pi/2+\theta]
\]
is attainable \cite{kurwig}. When $\theta=0$ and $\theta=\pi/4$ we recover the Cilleruelo and uniform measures respectively. 

In the case when the spectral measure of $\Psi $ is $\sigma_\theta$, the spectral measure $\sigma_{\theta,u}$  of the restriction $\psi_u$ is the projection of $\theta$. Unless $\theta=0$, this projection is continuous with respect to the Lebesgue measure and its support is given by the union of four intervals. 

Since the support is symmetric, it is enough to describe its positive part. It is easy to check that it is given by 
\begin{equation}
\label{supp}
\begin{aligned}[]
&
\begin{aligned}[]
&[\sin(u-\theta),\sin(u+\theta)] \\
&\cup[\cos(u+\theta),\cos(u-\theta)]
\end{aligned}
& \quad
\theta<u,\  u+\theta<\pi/4,
\\
&[\sin(u-\theta),\cos(u-\theta)]
&\quad
\theta<u,\ u+\theta\geq\pi/4,
\\
&[0,\sin(u+\theta)]\cup[\cos(u+\theta),1]
&\quad
\theta\geq u,\ u+\theta<\pi/4,
\\
&[0,1]
&\quad
\theta\geq u,\ u+\theta\geq\pi/4.
\end{aligned}
\end{equation}

For measures of this type  \cite{ffjnn1} gives also a lower bound on the persistence probability:
\begin{thm}[{\cite[Corollary 1]{ffjnn1}}]
\label{lb}
Consider a continuous stationary Gaussian process defined on $[0,L]$ with spectral measure that is compactly supported and has a non-trivial absolutely continuous component. Then there exists $L_0>0$ and  such that for every $L\geq L_0$ one has
\begin{equation}
\exp(-c'L^2)\leq\p(\Z=0)
\end{equation}
for some absolute constant $c'$.
\end{thm}


\begin{proof}[Proof of Proposition \ref{persb}]

First, let us consider the case $\theta<u$. According to \eqref{supp}, the spectral measure has a spectral gap and continuous with respect to the Lebesgue measure. By Theorems \ref{sgc} and \ref{lb}   we have 
\[
\exp(-c_1 L^2)\le \p(\Z=0)\le \exp(-c_2 L^2)
\]
for some positive constants $c_1$ and $c_2$ (that depend on $u$ and $\theta$).

Next, we consider the case $u\le\theta$. In this case the density of the spectral measure is a continuous non-vanishing function is a neighbourhood of the origin. For  such functions Theorem 1 of \cite{fefeni} states that
\[
\log \p(\Z=0) \approx L.
\]

\end{proof}

\subsection{Cilleruelo field}
The aim of Section \ref{seccil} is to prove Proposition \ref{percill}. 
\label{seccil}
We have the explicit formula for the restriction
\begin{multline}
\label{cilpro}
f(t)=f_u(t)=\frac{\sqrt{2}}{2}\left[b_1\cos\left(t\cos(u)\right)+b_2\cos\left(t\sin(u)\right)\right.\\\left.+c_1\sin\left(t\cos(u)\right)+c_2\sin\left(t\sin(u)\right)\right],
\end{multline}
with $b_1,b_2,c_1,c_2\sim\N(0,1)$ and i.i.d. The covariance function is
\[
\kappa_{u}(t)=\frac{1}{2}\left[\cos\left(t\cos(u)\right)+\cos\left(t\sin(u)\right)\right]
\]
and the spectral measure
\begin{equation}
\rho=\rho_u:=\frac{1}{4}(\delta_{\pm\cos(u)}+\delta_{\pm\sin(u)}).
\end{equation}
The behaviour of $\rho$ as $u$ varies already gives an indication that the cases $u=0,\pi/4$ are in some sense `special'. Indeed, in these cases $\rho$ is supported at three and two points respectively instead of four. Another indication in this sense is given by the following lemma.
\begin{lemma}[Relation between $f$ and its derivatives]
	\label{der}
	If $u=0$ then $f(0)$ is not determined by $f'(0),f''(0),\dots,f^{(k)}(0)$ for any $k$. If $u=\pi/4$ then
	\begin{equation*}
	f''(t)=-\frac{f(t)}{2}
	\end{equation*}
	($f$ is periodic).
	
	More generally, for every $0\leq u\leq\pi/4$, one has the relation
	\begin{equation*}
	\sin(u)^2\cos(u)^2f(t)+f''(t)+f^{(4)}(t)=0
	\end{equation*}
	for all $t$.
\end{lemma}

We will study the persistence probability of $f_u$, starting with the special cases $u=0,\pi/4$. Here we may actually compute the {\em distribution} of nodal intersections $\Z=\Z_f$.

For $u=0$ the expression \eqref{cilpro} simplifies to
\begin{equation*}
f_0(t)=\frac{\sqrt{2}}{2}[b_1\cos(t)+b_2+c_1\sin(t)]
\end{equation*}
and the covariance function is thus
\begin{equation*}
\kappa_0(t)=\cos^2\left(\frac{t}{2}\right).
\end{equation*}

\begin{lemma}[Bogomolny and Schmit \cite{bogsc7}]
	\label{bogsmi}
	For a Gaussian process defined on $[0,T]$ with covariance function $\kappa$,
	\begin{equation*}
	\p(\Z\text{ is even})=\frac{1}{2}+\frac{1}{\pi}\arcsin(\kappa(T)).
	\end{equation*}
\end{lemma}

\begin{prop}
	\label{distr0}
	If $u=0$ the r.v. $\Z=\Z_f$ is distributed as follows: for $0\leq L\leq 2\pi$,
	\begin{align*}
	\p(\Z=0)&=\frac{1}{4}\left(3-\frac{\sqrt{2}}{\pi}L\right)+\frac{1}{2\pi}\arcsin(\cos(L/2)^2),
	\\
	\p(\Z=1)&=\frac{1}{2}-\frac{1}{\pi}\arcsin(\cos(L/2)^2),
	\\
	\p(\Z=2)&=\frac{1}{4}\left(\frac{\sqrt{2}}{\pi}L-1\right)+\frac{1}{2\pi}\arcsin(\cos(L/2)^2);
	\end{align*}
	for $L\geq 2\pi$, letting $n=\lfloor L/2\pi\rfloor\in\mathbb{N}$ to be the integer part of $L/2\pi$,
\[
\begin{aligned}
	&\p(\Z=0)=1-\frac{\sqrt{2}}{2},
	\\
	&\p(\Z=2n)=-\frac{L}{2\pi\sqrt{2}}+\frac{n+1}{\sqrt{2}}-\frac{1}{4}+\frac{1}{2\pi}\arcsin(\cos(L/2)^2),
	\\
	&\p(\Z=2n+1)=\frac{1}{2}-\frac{1}{\pi}\arcsin(\cos(L/2)^2),
	\\
	&\p(\Z=2n+2)=\frac{L}{2\pi\sqrt{2}}-\frac{n}{\sqrt{2}}-\frac{1}{4}+\frac{1}{2\pi}\arcsin(\cos(L/2)^2).
	\end{aligned}
\]
\end{prop}
\begin{proof}
	For $u=0$ it is elementary to see that for $L=2\pi n$, $n\in\mathbb{N}$, one has
	\begin{equation}
	\label{banale}
	\p(\Z=0)=1-\sqrt{2}/2 \qquad \text{and} \qquad \p(\Z=2n)=\sqrt{2}/2.
	\end{equation}
Indeed, $f_0$ is periodic of period $2\pi$, hence in $[0,2\pi n]$ it has either $2L$ zeroes or none, depending on whether the straight line $y=b_2$ intersects the sine wave $y=b_1\cos(x)+c_1\sin(x)$. From the formula for the expectation of $\Z$ \eqref{expect} we then gather \eqref{banale}. We immediately deduce that $\p(\Z=0)=1-\sqrt{2}/2$ for every $L\geq 2\pi$.
	
	Let $0\leq L\leq 2\pi$. As $f_0$ is of period $2\pi$, it may have $0,1$ or $2$ zeroes. The respective probabilities are computed using Lemma \ref{bogsmi} together with the formula \eqref{expect} for the expected value of $\Z$. The remaining probabilities now follow from the periodicity of $f_0$.
\end{proof}

Now fix $u=\pi/4$ to obtain
\begin{equation*}
f_{\pi/4}(t)=\frac{\sqrt{2}}{2}[(b_1+b_2)\cos(t/\sqrt{2})+(c_1+c_2)\sin(t/\sqrt{2})]
\end{equation*}
which is just sine wave with a random phase shift and amplitude.
Its covariance function is
\begin{equation*}
\kappa_{\pi/4}(t)=\cos(t/\sqrt{2}).
\end{equation*}

\begin{prop}
	\label{distrpi4}
	Let $u=\pi/4$ and $n':=\lfloor L/\pi\sqrt{2}\rfloor$. For
	\begin{equation*}
	n'\pi\sqrt{2}\leq L\leq(n'+1)\pi\sqrt{2}
	\end{equation*}
	the r.v. $\Z=\Z_f$ is distributed as follows: 
\[
\begin{aligned}
	&\p(\Z=n')=1-\left\{\frac{L}{\pi\sqrt{2}}\right\},
	\\
	&\p(\Z=n'+1)=\left\{\frac{L}{\pi\sqrt{2}}\right\},
	\end{aligned}
\]
where $\{x\}:=x-\lfloor x\rfloor$ is the fractional part of a real number $x$.
\end{prop}
\begin{proof}
	The function $f_{\pi/4}$ is periodic of period $2\pi\sqrt{2}$, and has zeros at
	\begin{equation*}
	\frac{1}{\sqrt{2}}\left[-2\arctan\left(\frac{b_1+b_2}{c_1+c_2}\right)+2\pi k\right], \qquad k\in\mathbb{Z}.
	\end{equation*}
	In particular a.s. there is exactly one zero in each interval
	\begin{equation*}
	[n\pi\sqrt{2},(n+1)\pi\sqrt{2}], \qquad n=0,1,2,\dots.
	\end{equation*}
	The claims of the present Proposition are now all established thanks to Lemma \ref{bogsmi}.
\end{proof}

Propositions \ref{distr0} and \ref{distrpi4} yield in particular the persistence of $\Z_f$ in the cases $u=0,\pi/4$.
\begin{cor}
	\label{precper}
	If $u=0$ then
	\begin{equation}
	\label{noiu0}
	\p(\Z=0)=
	\begin{cases}
	\frac{1}{4}\left(3-\frac{\sqrt{2}}{\pi}L\right)+\frac{1}{2\pi}\arcsin(\cos(L/2)^2)
	&
	0\leq L\leq 2\pi,
	\\
	1-\frac{\sqrt{2}}{2}
	&
	L\geq 2\pi.
	\end{cases}
	\end{equation}
	If $u=\pi/4$ then
	\begin{equation}
	\p(\Z=0)=
	\begin{cases}
	1-L/\pi\sqrt{2}
	&
	0\leq L\leq\pi\sqrt{2},
	\\
	0
	&
	L\geq\pi\sqrt{2}.
	\end{cases}
	\end{equation}
\end{cor}
\begin{proof}
	This follows from Propositions \ref{distr0} and \ref{distrpi4}.
\end{proof}
Furthermore, for the Cilleruelo field along the directions $u=0,\pi/4$ one may compute the second factorial moment of $\Z$ exactly.
\begin{cor}
	Let $u=0$ and $n:=\lfloor L/2\pi\rfloor$. Then the second factorial moment of nodal intersections number $\Z$ is
	\begin{equation*}
	\E{\Z(\Z-1)}=\frac{(4n+1)L}{\pi\sqrt{2}}-2\sqrt{2}n(n+1)-\frac{1}{2}+\frac{1}{\pi}\arcsin(\cos(L/2)^2)
	\end{equation*}
	with asymptotic
	\begin{equation*}
	\E{\Z(\Z-1)}\sim\frac{L^2}{\pi^2\sqrt{2}}+O(L)
	\end{equation*}
	as $L\to\infty$. Let $u=\pi/4$ and $n':=\lfloor L/\pi\sqrt{2}\rfloor$. Then the second factorial moment of $\Z$ is given by
	\begin{equation*}
	\E{\Z(\Z-1)}=n'\left(n'-1+2\left\{\frac{L}{\pi\sqrt{2}}\right\}\right)
	\end{equation*}
	with asymptotic
	\begin{equation*}
	\E{\Z(\Z-1)}\sim\frac{L^2}{2\pi^2}+O(L)
	\end{equation*}
	as $L\to\infty$.
\end{cor}

\begin{proof}
	The second factorial moment
	\begin{equation*}
	\E{\Z(\Z-1)}:=\sum_{n=0}^{+\infty}n(n-1)\p(\Z=n)
	\end{equation*}
	is computed directly from the distribution of $\Z$ given by Propositions \ref{distr0} and \ref{distrpi4}.
\end{proof}

It is worth noting that for $u=0$ we have
\begin{equation*}
\text{Var}(\Z)=\frac{\sqrt{2}-1}{2\pi^2}L^2+O(L)
\end{equation*}
while for $u=\pi/4$
\begin{equation*}
\text{Var}(\Z)=\left\{\frac{L}{\pi\sqrt{2}}\right\}-\left\{\frac{L}{\pi\sqrt{2}}\right\}^2
\qquad\Rightarrow\qquad
0\leq\text{Var}(\Z)\leq\frac{1}{4}.
\end{equation*}
Indeed, in the computation of $\text{Var}(\Z)$ for $u=\pi/4$ several terms of order $L^2$ and in $L$ remarkably cancel out.  
It is natural to conjecture $\E{\Z(\Z-1)}\asymp L^2$ as $L\to\infty$ for the restriction of the Cilleruelo random wave to any straight line. Possibly the leading asymptotic is $\E{\Z(\Z-1)}\sim L^2\cos(u)/\pi^2\sqrt{2}$.

For $u=0,\pi/4$ computing explicitly the distribution of $\Z$ yields in particular the persistence of the Cilleruelo field along these directions. In case of general $u$, we have several upper and lower bounds for the persistence, as prescribed by Proposition \ref{percill}. We now complete the proof of this result.
\begin{proof}[Proof of Proposition \ref{percill}]
Statement \eqref{i} 
has already been proven in Corollary \ref{precper}.
Recall the expression \eqref{cilpro}
\begin{multline*}
f_u(t)=\frac{\sqrt{2}}{2}\left[b_1\cos\left(t\cos(u)\right)+b_2\cos\left(t\sin(u)\right)\right.\\\left.+c_1\sin\left(t\cos(u)\right)+c_2\sin\left(t\sin(u)\right)\right].
\end{multline*}

To show \eqref{ii}, assume $Lu\shortarrow{1}C_1<\pi/2$. For a parameter $k$ we write
\begin{equation*}
\p(\Z=0)>\p\left(|b_1|,|c_1|,|c_2|\leq 3k \ \wedge \ |b_2|>\frac{4k}{\cos\left(L\sin(u)\right)}\right).
\end{equation*}
Via Gaussian ball and tail estimates \cite[Lemma 3.14]{fefeni},
\begin{equation*}
\p(\Z=0)\geq C_7(k\exp(-k^2/2))^3\exp\left(\frac{-16k^2}{\cos\left(L\sin(u)\right)^2}\right)\geq C_0>0.
\end{equation*}

To show \eqref{iii}, let $Lu\shortarrow{1}\pi/2$: in this regime the computations for \eqref{ii} yield
\begin{equation*}
\p(\Z=0)\geq C_8\exp\left(\frac{-16k^2}{\cos\left(L\sin(u)\right)^2}\right)\geq C_2\exp\left(-\frac{C_3}{(Lu-\pi/2)^2}\right).
\end{equation*}

Statement \eqref{iv} follows directly from Proposition \ref{perub}. Indeed, unless $u=0$, one has a spectral gap and the assumptions of Proposition \ref{perub} are verified.

To show \eqref{v}, fix $0<u\leq\pi/4$ and $\epsilon>0$. If $x\in[-\pi/2+\epsilon,\pi/2-\epsilon]$ then $\cos(x)\geq\cos(\pi/2-\epsilon)>0$. With probability $1$ the Gaussian coefficients $b_1,b_2,c_1,c_2$ have different values, and w.l.o.g. $b_1$ dominates. With positive probability $b_1$ is big enough and $|b_2|,|c_1|,|c_2|$ small enough so that the Cilleruelo field $F(x)>0$ for all $x\in[-\pi/2+\epsilon,\pi/2-\epsilon]$. It follows that with probability $C_5(\epsilon)>0$ the process $f$ remains positive for $L\sin(u)<\pi-2\epsilon$.

To show \eqref{vi}, fix $0<u\leq\pi/4$ and $\epsilon>0$. With probability $1$ the Gaussian coefficients $b_1,b_2,c_1,c_2$ have different values, and w.l.o.g. $b_2$ dominates. With positive probability $f_u(0)>0$. Moreover, with positive probability $b_2$ is big enough and $|b_1|,|c_1|,|c_2|$ small enough so that $f_u((\pi/2+\epsilon)/\sin(u))$ has the same sign as $\cos(\pi/2+\epsilon)<0$. Therefore, with positive probability depending on $\epsilon$, $f_u(t)$ changes sign for $t\in[0,\pi/2+\epsilon]$. By continuity, with probability $C_6(\epsilon)>0$ the process $f_u$ has at least one zero provided $L\sin(u)>\pi/2+\epsilon$.

The last part \eqref{vii} is almost trivial. With probability one, the nodal lines of the Cilleruelo field $F$ are either $2\pi$ periodic vertical or horizontal lines. This means that they must intersect any interval of length $L$ in the direction $u$ as long as its projections onto both vertical and horizontal directions are longer than $2\pi$.

To be more precise we fix $0<u\leq\pi/4$. Almost surely $b_1^2+c_1^2\neq b_2^2+c_2^2$, and suppose for a moment that $b_1^2+c_1^2$ is larger. Again with probability $1$ we have $b_1,c_1\neq 0$ and w.l.o.g. $c_1>0$. Then for all $x_2$ one has 
\begin{equation*}
F(\alpha_1,x_2)>\sqrt{b_1^2+c_1^2}-\sqrt{b_2^2+c_2^2}>0
\end{equation*}
where
\begin{equation}
\label{a1}
    0<\alpha_1=2\arctan\left(\frac{\sqrt{b_1^2+c_1^2}-b_1}{c_1}\right)<\pi.
\end{equation}
Likewise, for all $x_2$ one has
\begin{equation*}
F(\alpha_2,x_2)<-\left(\sqrt{b_1^2+c_1^2}-\sqrt{b_2^2+c_2^2}\right)<0
\end{equation*}
where
\begin{equation}
\label{a2}
    \pi<\alpha_2=2\pi-2\arctan\left(\frac{\sqrt{b_1^2+c_1^2}+b_1}{c_1}\right)<2\pi.
\end{equation}
This determines nodal lines crossing the torus from top to bottom. Similarly, in case $b_2^2+c_2^2>b_1^2+c_1^2$ then there are nodal lines crossing the torus from left to right. It follows that if $L\sin(u)\geq 2\pi$ then a.s. the straight line $\mathcal{C}$ crosses the nodal line, i.e. $f_u$ has a zero.
\end{proof}


\section{Coupling}
\label{cou}
We will need a technical result which states that the  a field with small variance is small everywhere with probability which is close to one. The precise formulation is given in the following lemma.
\begin{lemma}[Lemma 3.12 of \cite{MuVa}]
\label{l:Kolmogorov}
There exists an absolute constant $c > 0$ such that, for every $C^1$-smooth planar Gaussian field $g$ and for all $R \ge  c$ and $h\ge \log R$
\begin{equation}
\P\left[ \|g\|_{C(B_R)}\ge mh\right]\le \exp(-h^2/c)
\end{equation} 
where
\[ 
m^2 =\sup_{x\in B_{R+1}}\sup_{|\alpha|\le 1}\E{(\partial^\alpha g)^2(x)}.
\]
\end{lemma}

Note that this statement is slightly different from that of Lemma 3.12 of \cite{MuVa}, but the proof is exactly the same. 

We are ready to prove Theorem \ref{discrete}.
\begin{proof}[Proof of Theorem \ref{discrete}]
Recall the definition \eqref{G} of $G$
\begin{equation}
\label{Gdiff}
G(x)=\frac{\sqrt{2}}{2}\left[\sum_{j=1}^{2M}b_j\cos(\langle x,y_j\rangle)+\sum_{j=1}^{2M}c_j\sin(\langle x,y_j\rangle)\right],
\end{equation}
where $N=4M$, $b_j,c_j\sim\N(0,2/N)$ i.i.d. and $y_j=(\cos(\varphi_j),\sin(\varphi_j))$,  $j=1,\dots,N$ with $|\varphi_j|\leq\epsilon$ for each $j=1,\dots,M$ for some fixed $\epsilon>0$. We defined the coupled field
\begin{align}
\label{Fdiff}
\notag
F(x)&=\frac{\sqrt{2}}{2}\left[(\sum_{|\varphi_j|\leq\epsilon}b_j)\cos(\langle x,(1,0)\rangle)+(\sum_{|\pi/2-\varphi_j|\leq\epsilon}b_j)\cos(\langle x,(0,1)\rangle)
\right.
\\
&\left.
+(\sum_{|\varphi_j|\leq\epsilon}c_j)\sin(\langle x,(1,0)\rangle)+(\sum_{|\pi/2-\varphi_j|\leq\epsilon}c_j)\sin(\langle x,(0,1)\rangle)\right]
\end{align}
where each summation has $M$ summands. It is easy to see that $F$ has the distribution of the Cilleruelo field. 

Subtracting \eqref{Fdiff} from \eqref{Gdiff},
\begin{align*}
	H(x):=(G(x)-F(x))\sqrt{2}&=\sum_{|\varphi_j|\leq\epsilon}b_j[\cos(\langle x,y_j\rangle)-\cos(\langle x,(1,0)\rangle)]\\&+\sum_{|\pi/2-\varphi_j|\leq\epsilon}b_j[\cos(\langle x,y_j\rangle)-\cos(\langle x,(0,1)\rangle)]\\&+\sum_{|\varphi_j|\leq\epsilon}c_j[\sin(\langle x,y_j\rangle)-\sin(\langle x,(1,0)\rangle)]\\&+\sum_{|\pi/2-\varphi_j|\leq\epsilon}c_j[\sin(\langle x,y_j\rangle)-\sin(\langle x,(0,1)\rangle)].
\end{align*}
We claim that $H(x)$, which is the difference between coupled fields is uniformly small with large probability.

To shorten and simplify the calculations we will estimate the first sum only. Estimates for other sum (or for all of them simultaneously) can be done in exactly the same way. Let us consider the function
\[
H_1(x)=\sum_{|\varphi_j|\leq\epsilon}b_j[\cos(\langle x,y_j\rangle)-\cos(\langle x,(1,0)\rangle)].
\]
This is a non-stationary Gaussian field with covariance
\[
K(x_1,x_2)=\frac{2}{N}\sum_{j=1}^M q(x_1,y_j)q(x_2,y_j),
\]
where
\[
q(x,y)=\cos(\langle x,y\rangle)-\cos(\langle x,(1,0)\rangle).
\]
Since cosine is a Lipschitz function with norm $1$ and $|y_j-(1,0)|\le \epsilon$, we have $q(x,y)\le |x|\epsilon$. By the same argument $|\nabla_x q(x,y)|\le |x|\epsilon$. Let us consider $m$ as in Lemma \ref{l:Kolmogorov} when $g=H_1$. The estimates above, imply that $m\le (R+1)\epsilon$. Considering $h=\log R$, Lemma \ref{l:Kolmogorov} gives
\[
\P\brb{\|H_1\|_{C(B_R)}\ge 2\epsilon R\log R}\le \exp(-\log^2 R/c).
\] 
By considering $h=R$ we get
\[
\P\brb{\|H_1\|_{C(B_R)}\ge 2\epsilon R^2 }\le \exp(-R^2/c).
\] 
Applying the same argument to the other terms (or estimating the covariance function of $H$) we complete the proof of the theorem.
\end{proof}

We end this section by proving Proposition \ref{perfg}.

\begin{proof}[Proof of Proposition \ref{perfg}]
Recall the expression for the restricted Cilleruelo field $f$ \eqref{cilpro}. Let $0<u\leq\pi/4$ and $L\sin(u)\geq 2\pi$. Suppose that 
\begin{equation*}
    \left|\sqrt{b_1^2+c_1^2}-\sqrt{b_2^2+c_2^2}\right|>2\epsilon L \log L
\end{equation*}
and moreover that
\begin{equation*}
    |G(x)-F(x)|<2\epsilon L \log L
\end{equation*}
for every $x\in B(0,L)$. Then the proof of Proposition \ref{percill} \eqref{vii} tells us that for every $x_2$ we have $G(\alpha_1,x_2)>0$ and $G(\alpha_2,x_2)<0$ with $\alpha_1,\alpha_2$ as in \eqref{a1} and \eqref{a2} respectively. As $L\sin(u)\geq 2\pi$, it follows that the process $G$ has a zero in $[0,L]$. 

By the argument above the event $\Z_g=0$ is inside the union of two events:
\[
\left|\sqrt{b_1^2+c_1^2}-\sqrt{b_2^2+c_2^2}\right|<2\epsilon L \log L 
\]
which means that $F$ is close to zero, and
\[
\|G-F\|\ge 2\epsilon L \log L 
\]
which means that $F$ and $G$ differ too much.

The probability of the first event is 
\[
\p\br{|X_1-X_2|<2\epsilon L \log L}={\sqrt{\pi}}\epsilon L \log L +O( \epsilon L \log L)^2,
\]
where $X_1$ and $X_2$ are independent $\chi(2)$ random variables.

By Theorem \ref{discrete} the probability of the second event is bounded by 
\begin{equation}
\label{eq:div}
\exp(-(\log^2 L)/c).
\end{equation}
 
Combining these estimates we have
\[
\p(\Z_g=0)\le \exp(-(\log^2 L)/c)+{\sqrt{\pi}}\epsilon L \log L +O( \epsilon L \log L)^2.
\]

For the rest of this proof fix $u=0$. We apply the same logic: if the restriction of $F$ is not too small and $F-G$ is small, then the restriction of $G$ is also of constant sign. 

The event $f>2\epsilon L \log L$ is equivalent  to $b_2>\sqrt{b_1^2+c_1^2}+2\epsilon L \log L$. As $b_2\sim\mathcal{N}(0,1)$ and $\sqrt{b_1^2+c_1^2}\sim\chi(2)$ are independent, we get
\begin{equation}
    \label{es5}
    \p(f>2\epsilon L \log L)=\frac{1}{2}\left(1-\frac{\sqrt{2}}{2}-\frac{2\epsilon L \log L }{\sqrt{2\pi}}+O(\epsilon L \log L )^2\right)
\end{equation}
via a routine computation. Combining this with the bound \eqref{eq:div}
\[
\p(\Z_g=0)\geq 
1-\frac{\sqrt{2}}{2}-\exp(-(\log^2 L)/c)-\frac{2\epsilon L \log L }{\sqrt{2\pi}}+O(\epsilon L \log L )^2.
\]

\end{proof}

\addcontentsline{toc}{section}{References}

\bibliography{bibfile}{}
\bibliographystyle{plain}

\Addresses

\end{document}